\documentclass[letterpaper, 10pt, journal, twocolumn, final]{IEEEtran}

\usepackage{amsthm}
\usepackage{soul}
\usepackage[normalem]{ulem}
\usepackage[hidelinks]{hyperref}
\usepackage{bookmark}
\usepackage[latin1]{inputenc}
\usepackage[english]{babel}

\usepackage[dvipsnames]{xcolor}
\usepackage{amsfonts,amssymb,amsmath,color}
\usepackage[draft]{graphicx}
\usepackage[noadjust]{cite}
\usepackage{placeins}
\usepackage{graphicx}
\usepackage{algorithm}
\usepackage[]{algpseudocode}
\usepackage{dsfont}

\usepackage{pgfplots,tikz}
\usetikzlibrary{shapes,arrows,calc}

\usepackage{scalefnt}
\let\oldtextsc\textsc
\renewcommand{\textsc}[1]{\oldtextsc{\scalefont{0.9}#1}}

\newcommand\oprocendsymbol{\hbox{$\square$}}
\newcommand\oprocend{\relax\ifmmode\else\unskip\hfill\fi\oprocendsymbol}

\newtheorem{theorem}{Theorem}[section]
\newtheorem{proposition}[theorem]{Proposition}

\newtheorem{lemma}[theorem]{Lemma}

\newtheorem{assumption}[theorem]{Assumption}

\makeatletter
\newcommand{\StatexIndent}[1][3]{%
  \setlength\@tempdima{\algorithmicindent}%
  \Statex\hskip\dimexpr#1\@tempdima\relax}
\makeatother

\renewcommand{\inf}{\operatornamewithlimits{inf\vphantom{p}}}

\renewcommand{\lim}{\operatornamewithlimits{lim\vphantom{p}}}

\newcommand{\real}{{\mathbb{R}}}

\renewcommand{\natural}{{\mathbb{N}}}

\newcommand{\until}[1]{\{1,\ldots,#1\}} 
\newcommand{\interv}[1]{1,\ldots,#1} 

\newcommand{\map}[3]{#1: #2 \rightarrow #3}

\newcommand{\argmax}{\mathop{\rm argmax}}

\newcommand{\find}{\text{find}}
\newcommand{\subj}{\textnormal{subj. to}}
 
\newcommand{\nbrs}{\mathcal{N}}
\newcommand{\subgrad}{\widetilde{\nabla}}

\newcommand{\0}{\mathbf{0}}
\newcommand{\1}{\mathds{1}}

\newcommand{\EE}{\mathcal{E}}

\newcommand{\GG}{\mathcal{G}}

\newcommand{\tp}{\tilde{p}}

\newcommand{\oracle}{\textsc{Oracle}}
\newcommand{\ymax}{y^\textsc{max}}
\newcommand{\ymin}{y^\textsc{min}}

\newcommand{\xmax}{x^\textsc{max}}
\newcommand{\rhomax}{\rho^\textsc{max}}
\newcommand{\xmin}{x^\textsc{min}}
\newcommand{\rhomin}{\rho^\textsc{min}}

\newcommand{\tx}{\tilde{x}}

\newcommand{\trho}{\tilde{\rho}}

\newcommand{\by}{\bar{y}}

\graphicspath{{figs/}}

\def \algname/{Distributed Primal Decomposition with Costs Learning}
\def \algacronym/{DPD-Learn}

\begin{document}

\title{\huge Constraint-coupled Optimization with Unknown Costs: \\ A Distributed Primal Decomposition Approach}

\author{Andrea Camisa, %
  Alessia Benevento,
  Giuseppe Notarstefano %
  \thanks{A. Benevento is with the Department of Engineering,
  University of Salento, Lecce, Italy.
  \texttt{alessia.benevento@unisalento.it.}
  A. Camisa and G. Notarstefano are with the Department of Electrical,
  Electronic and Information Engineering, University of Bologna, Bologna, Italy.
  \texttt{\{a.camisa, giuseppe.notarstefano\}@unibo.it}.
  }
  \thanks{This result is part of a project that has received funding from the European 
  Research Council (ERC) under the European Union's Horizon 2020 research 
  and innovation programme (grant agreement No 638992 - OPT4SMART).
  }
}

\maketitle

\begin{abstract}
	In this paper, we present a distributed algorithm for solving convex,
	constraint-coupled, optimization problems over peer-to-peer networks.
	We consider a network of processors that aim to cooperatively minimize
	the sum of local cost functions, subject to individual constraints and
	to global coupling constraints. The major assumption of this work is that
	the cost functions are unknown and must be learned online.
	We propose a fully distributed algorithm, based on a primal decomposition
	approach, that uses iteratively refined data-driven estimations of
	the cost functions over the iterations. The algorithm is scalable and
	maintains private information of agents.
	We prove that, asymptotically, the distributed algorithm provides the
	optimal solution of the problem even though the true cost functions
	are never used within the algorithm.
	The analysis requires an in-depth exploration of the primal decomposition
	approach and shows that the distributed algorithm can be thought of
	as an epsilon-subgradient method applied to a suitable reformulation
	of the original problem.
	Finally, numerical computations corroborate the theoretical findings
	and show the efficacy of the proposed approach.
\end{abstract}

\begin{IEEEkeywords}
  Distributed Optimization, Cost Function Estimation, Constraint-Coupled Optimization
\end{IEEEkeywords}

\section{Introduction}

Last decades have seen an increasing interest in distributed optimization over
networks, due to the ubiquitous presence of networked structures~\cite{bullo2019lectures}
such as social influence networks~\cite{sasso2019interaction}, wireless sensor
networks~\cite{rabbat2005quantized} or multi-robot
systems~\cite{bullo2009distributed,cortes2017coordinated}.
In a distributed optimization framework, agents in a network aim to
\emph{cooperatively} minimize the sum of objective functions,
each one assigned to an agent of the network, subject to constraints.
Many works have concentrated on cost-coupled optimization, where the cost function
is the sum of local functions depending on a common decision
variable. An exemplary, non-exhaustive list of works addressing this problem
set-up is~\cite{nedic2009distributed, di2016next, varagnolo2015newton,mateos2017distributed}.
Follow-up works have addressed proposed a more challenging optimization
scenario, which we call constraint-coupled, in which each
local cost function depends on a local variable that is subject to a
local constraint, however it is also necessary to satisfy global coupling constraints
involving all the decision variables.
While the previous problem set-up is more related to estimation and machine learning,
the latter is more relevant to distributed control applications, see~\cite{notarstefano2019distributed}.
The solution of constraint-coupled optimization problems is typically achieved
by duality-based decomposition approaches~\cite{palomar2006tutorial,giselsson2013accelerated,notarnicola2019constraint,camisa2020distributed},
possibly using also penalty approaches~\cite{dinh2013dual} or smoothing techniques~\cite{tran2016fast}.

In this paper, we focus on a challenging constraint-coupled scenario in which
the local cost functions are unknown and have to be estimated online.
The problem of estimating unknown functions from observed samples has
been extensively studied and a number solutions have been proposed, see
e.g.~\cite{simonetto2020smooth, williams2006gaussian,schulz2018tutorial}.
The joint estimation and optimization represent several real-world situations, often occurring
in complex systems scenarios, such as in presence of autonomous underwater
vehicles (AUV) %
or mobile teams of robots~\cite{todescato2017multi,benevento2020multi}.
Similar problems exist even in the context of biological networks, where, e.g.,
groups of animals cooperate with each other for reaching a common goal,
such as locating food sources or avoinding predators. %

The contributions of this paper are as follows. We consider a distributed constraint-coupled
optimization set-up with unknown cost functions to be learned online.
Under the assumptions that the agents are able to build more and more
refined estimates of their cost functions, we propose a novel,
distributed algorithm to solve the problem \emph{exactly}.
The algorithm is inspired to a distributed primal decomposition approach
for constraint-coupled optimization~\cite{camisa2019distributed,camisa2020distributed},
where the algorithm has been suitably modified to account for the online
cost estimation mechanism.
The resulting scheme is a three-step procedure where each agent first
obtains an updated version of the cost estimation, then it solves
a local version of the original problem with the true cost function replaced
by an estimated version, and finally updates a local state after exchanging
dual information with neighbors. Interestingly, this algorithm scales with
the size of the network and avoids that private information of agents
(such as the constraints or the estimated solution) is disclosed. 
We prove that the algorithm asymptotically solves the problem even though
the true cost is never used. To obtain this result, we rely on an in-depth
analysis of the primal decomposition approach and on the consequences
of not using the true cost function. We show that the net effect of using
the estimated costs is that the distributed algorithm can be reinterpreted
as an \emph{epsilon-subgradient} method applied to a suitably obtained
reformulation of the original problem.
Finally, the theoretical findings are corroborated with numerical computations.

The paper is organized as follows. In Section~\ref{sec:prob_statement}
we describe the problem set-up. In Section~\ref{sec:distributed_algorithm}
we present our distributed algorithm. To analyze the algorithm, we first
present preliminary results in Section~\ref{sec:analysis_preliminary}
and then we conclude the analysis in Section~\ref{sec:analysis}.
Numerical computations are provided in Section~\ref{sec:simulations}.

\section{Problem Statement}
\label{sec:prob_statement}
In this section, we formalize the problem set-up studied in the paper
together with the needed assumptions.

\subsection{Distributed Constraint-coupled Optimization}
We consider a network of $N$ agents that must solve the
optimization problem
\begin{align}
\begin{split}
  \min_{x_1, \ldots, x_N} \: & \: \sum_{i=1}^N f_i(x_i)
  \\
  \subj \: & \: \sum_{i=1}^N g_i(x_i) \le b,
  \\
  & \: x_i \in X_i, \hspace{1cm} i \in \until{N},
\end{split}
\label{eq:problem_original}
\end{align}
where, for all $i \in \until{N}$, $x_i \in \real^{n_i}$ ($n_i \in \natural$) is the
$i$-th optimization variable with constraint set $X_i \subset \real^{n_i}$,
$\map{f_i}{\real^{n_i}}{\real}$ is the $i$-th cost function and
$\map{g_i}{\real^{n_i}}{\real}$ is the $i$-th contribution to the coupling
constraint, with right-hand side $b \in \real$.
Due to the presence of the constraint
$\sum_{i=1}^N g_i(x_i) \le b$
the optimization variables are entangled and a distributed solution
of the problem is not trivial.
For this reason, we term the structure of problem~\eqref{eq:problem_original}
\emph{constraint coupled}.

The following two assumptions guarantee that
\emph{(i)} the optimal cost of problem~\eqref{eq:problem_original} is finite
and at least one optimal solution exists,
\emph{(ii)} duality arguments are applicable.
\begin{assumption}[Convexity and compactness]
\label{ass:problem}
  For all $i \in \until{N}$, the set $X_i$ is non-empty, convex and compact, the
  function $f_i$ is convex and each component of $g_i$ is a convex function.
  \oprocend
\end{assumption}
\begin{assumption}[Slater's constraint qualification]
\label{ass:slater}
  There exist $\bar{x}_1 \in X_1, \ldots, \bar{x}_N \in X_N$ such that
  $\sum_{i=1}^N g_i(\bar{x}_i) < b$.
  \oprocend
\end{assumption}

We assume that each node $i$ does not know the entire problem information.
In particular, we assume it only knows the local constraint $X_i$, its contribution
$g_i$ to the coupling constraint and the right-hand side $b$.

The exchange of information among $N$ agents occurs according to a fixed
communication model. We use $\GG = (V, \EE)$ to indicate the undirected,
connected graph describing the network, where $V = \until{N}$ is the set of vertices
and $\EE$ is the set of edges.
If $(i,j)\in \EE$, then agent $i$ can communicate with agent $j$ and viceversa.
We use $\nbrs_i$ to indicate the set of neighbors of agent $i$ in $\GG$, i.e.,
$\nbrs_i = \{ j \in V | (i,j) \in \EE \}$.
The assumption of static graph can be relaxed to handle the more general case of
time-varying graphs as described in \cite{camisa2020distributed}.
However, this is not the main focus of this work, thus we prefer to maintain
the assumption of static network to keep the discussion simple.

\subsection{Unknown Cost Functions and their Estimation}
The main feature of the scenario considered in this paper is that the cost functions are \emph{not}
known in advance and must be estimated online. This challenging assumption can model,
for instance, situations in which evaluation of the cost function is computationally intensive
and can be done only for a small number of points. For this reason, we assume that agents
are equipped with an estimation mechanism that progressively refines their knowledge
of the objective functions.

We model the estimation mechanism as a black-box \emph{oracle} that can be queried
to provide estimations of the cost function. Since each agent $i$ has its own cost function,
we assume that each agent has its own instance of the oracle providing estimated
versions of the cost function $f_i(x_i)$. As these oracles will be embedded within the
iterative distributed algorithm introduced in Section~\ref{sec:distributed_algorithm},
we denote by $f_i^t(\cdot) = \oracle(i, t)$ the output of oracle $i$ at an iteration
$t \in \natural$. We do not impose a specific estimation mechanism, so that each agent
$i$ can use the most appropriate method depending on the cost function at hand.
The only assumption on the oracles needed for the distributed algorithm is formalized next.
\begin{assumption}[Oracles]
\label{ass:oracles}
  For each agent $i \in \until{N}$, the estimated functions $f_i^t(\cdot) = \oracle(i, t)$
  converge uniformly to the true cost function $f_i(\cdot)$.
  \oprocend
\end{assumption}

The main goal of the work is to propose a distributed algorithm such that the group
of agents simultaneously estimate the cost function and find a global optimal solution of
problem~\eqref{eq:problem_original}. To this end, the distributed algorithm makes use of
the estimated cost functions in place of the original ones, but nevertheless it will
be able to solve problem~\eqref{eq:problem_original} exactly.

\section{\algname/}
\label{sec:distributed_algorithm}
In this section, we describe the proposed distributed algorithm for
solving \eqref{eq:problem_original}.
Let $t \in \natural$ be an iteration index and let each agent maintain
a local estimate $x_i^t \in \real^{n_i}$ of the solution and
a local allocation $y_i^t \in \real$ of the coupling constraints.
Initially, the local allocation is initialized such that
$\sum_{i=1}^N y_i^0 = b$, e.g., $y_i^0 = \frac{b}{N}$.
At each iteration $t$, each agent $i$ first queries the oracle
to obtain a new, more refined estimation of the cost $f_i^t(\cdot)$.
Then, it solves a small, local optimization problem using
the estimated cost function and in particular it computes
both a primal solution $x_i^t$ and a dual solution $\mu_i^t$.
Finally, after exchanging $\mu_i^t$ with its neighbors,
the agent updates the local allocation $y_i^t$.
Algorithm~\ref{alg:algorithm} summarizes the distributed algorithm
from the perspective of node $i$, where we denote by $\alpha^t \ge 0$
the step-size, while the notation ``$\mu_i :$'' means that
the Lagrange multiplier $\mu_i$ is associated to the constraint
$g_i(x_i) \leq y_i^t+\rho_i$.
\begin{algorithm}[htbp]
\floatname{algorithm}{Algorithm}

  \begin{algorithmic}[0]
    
    \Statex \textbf{Initialization}: $y_i^0 \in \real$ such that $\sum_{i=1}^N y_i^0 = b$
    \medskip
    
    \Statex %
    \textbf{For} $t = 0, 1, 2, \ldots$
    \medskip

      \StatexIndent[0.75]
      \textbf{Query oracle} and obtain local cost estimation $f_i^t(\cdot) = \oracle(i, t)$
      \StatexIndent[0.75]
      \textbf{Compute} $((x_i^t, \rho_i^t), \mu_i^t)$ as a primal-dual pair of
      \begin{align}
			\begin{split}
			    \min_{x_i, \rho_i} \hspace{0.8cm} &\: f_i^t(x_i) + M \rho_i
			    \\
			    \subj \:
			    \:\: \mu_i :
			    & \: g_i(x_i) \leq y_i^t+\rho_i
			    \\
			    & \: x_i \in X_i
			\end{split}
      \label{eq:alg_local_prob}
      \end{align}
      
      \StatexIndent[0.75]
      \textbf{Gather} $\mu_j^t$ for $j \in \nbrs_i$ and update
      \begin{align}
        y_i^{t+1} = y_i^t + \alpha^t \sum_{j \in \nbrs_i^t} \big( \mu_i^t - \mu_j^t \big)
      \label{eq:alg_update}
      \end{align}

  \end{algorithmic}
  \caption{\algacronym/}
  \label{alg:algorithm}
\end{algorithm}

Let us outline some comments on the algorithm. Note that in
Algorithm~\ref{alg:algorithm} the true cost function never appears.
Instead, in the local problems~\eqref{eq:alg_local_prob},
only the estimated versions of the cost function are used.
When the actual cost function $f_i$ is complex, using a surrogate function
in place of the true one can significantly reduce the computational cost of
solving that step of the algorithm.
We also highlight that the amount of computation stays
constant as the size $N$ of the network grows, therefore the
algorithm is scalable. Moreover, since each agent only
exchanges dual information with the neighbors, no private
information (such as the local solution estimates $x_i^t$ or the
local constraints $X_i$) is disclosed during the algorithmic evolution.

Once again, we point out that there is no constraint on the
technique to be used for the learning part (i.e. the oracle),
as long as Assumption~\ref{ass:oracles} is met.
As $t$ increases, the estimated function $f_i^t$ approaches the
true function $f_i$, therefore the algorithm is expected to
recover some kind of consistency on the long run.
This is indeed the case as we will formally show in the next sections.
The main theoretical result is represented by Theorem~\ref{thm:convergence},
which formalizes the assumptions under which Algorithm~\ref{alg:algorithm}
solves problem~\eqref{eq:problem_original} to optimality.

\section{Algorithm Analysis: Preliminary Results}
\label{sec:analysis_preliminary}

In this section, we introduce preliminary results that are required
for the subsequent analysis. We first recall the primal decomposition
approach and then we elaborate novel results on the combination
of primal decomposition with the cost estimation mechanism.
To keep the notation light, in the proofs of this section we omit
the index $i$.

\subsection{Review of Primal Decomposition}
\label{subsec:review}

The main tool that we use to solve problem~\eqref{eq:problem_original}
in a distributed fashion is the primal decomposition
approach~\cite{silverman1972primal,bertsekas1999nonlinear}.
In particular, we consider a variant of this decomposition scheme
that is combined with a so-called relaxation approach.
This variant is particularly suited for distributed computation
and is described in~\cite{notarnicola2019constraint,camisa2020distributed}.
Let us briefly recall the main results that are needed for the
forthcoming analysis.

The approach is based on the following relaxation of the original
problem~\eqref{eq:problem_original},
\begin{align}
\begin{split}
  \min_{\substack{x_1, \ldots, x_N, \\ \rho_1, \ldots, \rho_N}} \: & \: \sum_{i=1}^N \left( f_i(x_i) + M \rho_i \right)
  \\
  \subj \: & \: \sum_{i=1}^N g_i(x_i) \le b + \sum_{i=1}^N \rho_i,
  \\
  & \: x_i \in X_i, \:\: \rho_i \ge 0 \hspace{0.5cm} i \in \until{N},
\end{split}
\label{eq:problem_relaxed}
\end{align}
where $M > 0$ is a scalar parameter. The additional variables $\rho_i$ permit
a violation of the coupling constraint (in this sense we say
that~\eqref{eq:problem_relaxed} is a relaxed version of~\eqref{eq:problem_original}),
while penalizing the cost through the term $M \sum_{i=1}^N \rho_i$.
For a sufficiently large $M > 0$, the optimal solutions of~\eqref{eq:problem_relaxed}
have zero violation, recovering the original solutions of~\eqref{eq:problem_original}%
\footnote{As highlighted in \cite{camisa2020distributed},
the formulation~\eqref{eq:problem_relaxed}, although equivalent to the original
one~\eqref{eq:problem_original}, is more convenient for the distributed solution
of the problem.}.
Problem~\eqref{eq:problem_relaxed} is then decomposed hierarchically
using the primal decomposition technique~\cite{silverman1972primal}.
Formally, for all $i \in \until{N}$ we introduce local \emph{allocation} vectors
$y_i \in \real$, which capture the utilization of coupling constraint by each agent
$i$, and define a \emph{master problem},
\begin{align}
\label{eq:primal_decomp_master}
\begin{split}
  \min_{y_1, \ldots, y_N} \: & \: \sum_{i=1}^N p_i(y_i)
  \\
  \subj \: & \: \sum_{i=1}^N y_i = b,
\end{split}
\end{align}
where, for all $i \in \until{N}$, the function $\map{p_i}{\real}{\real}$
is the optimal cost of the $i$-th \emph{subproblem},
\begin{align}
\label{eq:primal_decomp_subprob}
\begin{split}
  p_i(y_i) \triangleq \min_{x_i, \rho_i} \: & \: f_i(x_i) + M \rho_i
  \\
  \subj \: & \: g_i(x_i) \le y_i + \rho_i
  \\
  & \: x_i \in X_i, \:\: \rho_i \ge 0,
\end{split}
\end{align}
which is parametric in the value of the right-hand side of the constraint.
The equivalence between the original problem~\eqref{eq:problem_original}
and problems~\eqref{eq:primal_decomp_master}--\eqref{eq:primal_decomp_subprob}
is summarized in the next lemma.
\begin{lemma}[{\cite[Lemma 2.7, Lemma 2.8]{camisa2020distributed}}]
\label{lemma:relaxation_primal_decomposition}
	Let problem~\eqref{eq:problem_original} be feasible and let
	Assumptions~\ref{ass:problem} and~\ref{ass:slater} hold.
	Then, for a sufficiently large $M > 0$, problem \eqref{eq:primal_decomp_master}
	and \eqref{eq:problem_original} are equivalent, in the sense that
	\emph{(i)} the optimal costs are equal, \emph{(ii)} if
	$(x_1^\star, \ldots, x_N^\star)$ is an optimal solution
	of~\eqref{eq:problem_original} and $(y_1^\star, \ldots, y_N^\star)$
	is an optimal solution of~\eqref{eq:primal_decomp_master},
	then $(x_i^\star, 0)$ is an optimal solution of~\eqref{eq:primal_decomp_subprob}
	(with $y_i = y_i^\star$ for all $i \in \until{N}$).
\oprocend
\end{lemma}

We also recall two useful results on the functions $p_i(y_i)$.
Throughout the analysis, we will use the superscript $\cdot^\prime$
to indicate both sub-derivatives and derivatives.
Whether the symbol denotes a derivative or a subderivative
will always be clear from the context.
The first lemma provides an operative way to compute
subderivatives of $p_i(y_i)$.
\begin{lemma}[{\cite[Section 5.4.4]{bertsekas1999nonlinear}}]
  Let $y_i \in \real$ be given and let Assumptions~\ref{ass:problem}
  and~\ref{ass:slater} hold.
  Then, a subderivative of $p_i$ at $y_i$,
  denoted $p_i^\prime(y_i)$, can be computed as
  \begin{align}
    p_i^\prime(y_i) = -\mu_i(y_i),
  \end{align}
  where $\mu_i(y_i)$ denotes a Lagrange multiplier of
  problem~\eqref{eq:primal_decomp_subprob} associated
  to the constraint $g_i(x_i) \leq y_i$.
  \oprocend
\label{lemma:primal_function_subgradient}
\end{lemma}
Note that Lagrange multipliers can be computed as dual optimal
solutions of problem~\eqref{eq:primal_decomp_subprob}.
The second lemma provides bounds on the subderivatives of
$p_i(y_i)$.
\begin{lemma}
\label{lemma:p_bounded_negative_subgrad}
  For all $i \in \until{N}$ and for all $y_i \in \real$, the subderivatives
  of $p_i$ satisfy $-M \le p_i^\prime(y_i) \le 0$.
  \oprocend
\end{lemma}
The proof of Lemma~\ref{lemma:p_bounded_negative_subgrad} relies on
duality-based arguments similar to the ones in~\cite{notarnicola2019constraint}.
For completeness, we report it in Appendix~\ref{app:proof_primal_func_subgradient}.

\subsection{Properties of the Primal Functions}
Central to the analysis is the role of the functions $p_i(y_i)$.
In this section, we explore more deeply their structure.
As explained in~\cite[Section 5.4.4]{bertsekas1999nonlinear}, such function is
also called \emph{primal function} of the optimization
problem~\eqref{eq:primal_decomp_subprob} and has many
important properties (such as convexity).
Let us define for all $i \in \until{N}$ the scalars
\begin{subequations}
\begin{align}
  \ymin_i &\triangleq \min_{x_i \in X_i} \: g_i(x_i),
  \label{eq:y_min_def}
  \\
  \ymax_i &\triangleq \max_{x_i \in X_i} \: g_i(x_i),
  \label{eq:y_max_def}
\end{align}%
\end{subequations}
from which it directly follows that any locally feasible solution
$x_i \in X_i$ satisfies $\ymin_i \le g_i(x_i) \le \ymax_i$.

The next important lemma
regards the structure of the primal functions, which is
graphically represented in Figure~\ref{fig:p_prop}.
Intuitively, the numbers $\ymin_i$, $\ymax_i$ represent the minimum and
maximum resource that each agent $i$ can use.
As the allocation $y_i$ ranges from $\ymin_i$ to $\ymax_i$,
the optimal cost of the subproblem~\eqref{eq:primal_decomp_subprob}
decreases since the constraint $g_i(x_i) \le y_i + \rho_i$ becomes
less and less stringent. Eventually, for allocations greater than
$\ymax_i$, the cost cannot be further improved and $p_i(y_i)$
becomes constant. Instead, if $y_i \le \ymin_i$, optimal solutions
to problem~\eqref{eq:primal_decomp_subprob}
must compensate for the gap $\ymin_i - y_i$ with an appropriate
choice of $\rho_i$. The cost penalty $M \rho_i$ gives rise to the
linear behavior.
\begin{figure}[htbp]\centering
  \includegraphics[scale=1]{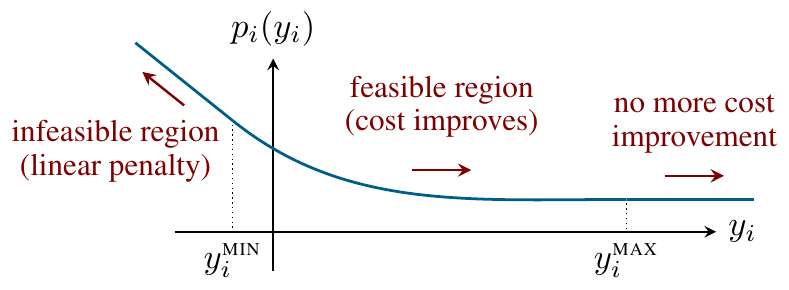}
  \caption{
    Illustration of Lemma~\ref{lemma:primal_function_properties}.
    See the text for details.
  }
\label{fig:p_prop}
\end{figure}
\begin{lemma}
\label{lemma:primal_function_properties}
  For all $i \in \until{N}$, the primal function $p_i(y_i)$ satisfies
  the following properties:
  \begin{enumerate}
    \item $p_i(y_i) = p_i(\ymax_i)$ for all $y_i \ge \ymax_i$;
    \item $p(y_i) = -M y_i + q_i$ for all $y_i \le y_i^\textsc{min}$,
  \end{enumerate}
  with $q_i = p_i(\ymin_i) + M \ymin_i$.
\end{lemma}
\begin{proof}
  Let us show \emph{(i)}. Let $\by \ge \ymax$ and let $(\xmax, \rhomax)$ be an
  optimal solution of problem~\eqref{eq:primal_decomp_subprob} with $y = \ymax$.
  To prove that $p(\by) = p(\ymax)$, we must prove that $(\xmax, \rhomax)$ is an
  optimal solution of problem~\eqref{eq:primal_decomp_subprob} when $y = \by$.
  By construction, it holds
  \begin{align*}
    g(\xmax) \le \ymax + \rhomax \le \by + \rhomax
  \end{align*}
  thus $(\xmax, \rhomax)$ is a feasible solution. Suppose that it is not optimal,
  then there exists $(\tx, \trho)$ such that $\tx \in X$, $\trho \ge 0$,
  $g(\tx) \le \by + \trho$ and
  \begin{align*}
    f(\tx) + M\trho < f(\xmax) + M \rhomax,
  \end{align*}
  i.e., $(\tx, \trho)$ has a lower cost than $(\xmax, \rhomax)$.
  However, by~\eqref{eq:y_max_def} and $\trho \ge 0$ we have
  \begin{align*}
    g(\tx) \le \max_{x_i \in X_i} g_i(x_i) = \ymax \le \ymax + \trho
  \end{align*}
  and $(\tx, \trho)$ would be a feasible solution
  for problem~\eqref{eq:primal_decomp_subprob} with $y = \ymax$
  with a cost lower than $(\xmax, \rhomax)$, contradicting the assumption
  that $(\xmax, \rhomax)$ is optimal.
  
  Now we prove \emph{(ii)}.
  Let $\by \le y^\textsc{min}$ and let $(\xmin, \rhomin)$ be optimal solution of
  problem~\eqref{eq:primal_decomp_subprob} with $y = \ymin$.
  It holds $p(\ymin) = f(\xmin) + M\rhomin$ and $g(\xmin) \leq \ymin+\rhomin$.
  The goal is to show that
  \begin{align*}
    p(\by)
    &= p(\ymin) + M (\ymin - \by)
    \\
    &= f(\xmin) + M(\rhomin + \ymin - \by),
  \end{align*}
  i.e. that $(\xmin, \rhomin + \ymin-\by)$ is an
  optimal solution of problem~\eqref{eq:primal_decomp_subprob} with
  $y = \by$.
  By using the assumption on $(\xmin, \rhomin)$ and the fact that
  $\ymin-\by\ge 0$ (by~\eqref{eq:y_min_def}), we can immediately show feasibility,
  \begin{align*}
    g(\xmin)
    &\le \ymin + \rhomin
    \\
    &\le \ymin + \rhomin + \ymin - \by .
  \end{align*}
  Suppose that $(\xmin, \rhomin + \ymin-\by)$ is not
  optimal for problem~\eqref{eq:primal_decomp_subprob} with
  $y = \by$. Then, there exists $(\tx, \trho)$ such that $\tx \in X$,
  $\trho \ge 0$, $g(\tx) \le \by + \trho$ and
  \begin{align*}
    f(\tx) + M\trho
    &< f(\xmin) + M(\trho + \ymin-\by)
    \\
    &= p(\ymin) + M(\ymin-\by),
  \end{align*}
  from which it follows that $f(\tx) + M(\trho + \ymin-\by) < p(\ymin)$,
  i.e., the vector $(\tx, \trho + \ymin-\by)$ has a lower cost than
  $(\xmin, \rhomin)$. Moreover, using again $\ymin - \by \ge 0$,
  we obtain
  \begin{align*}
    g(\tx)
    \le \by + \trho
    \le \by + \trho + \ymin - \by,
  \end{align*}
  from which it follows that $(\tx, \trho + \ymin-\by)$ is a feasible solution
  for problem~\eqref{eq:primal_decomp_subprob} with $y = \by$
  with a cost lower than $(\xmin, \rhomin)$, contradicting the assumption
  that $(\xmin, \rhomin)$ is optimal. 
\end{proof}

In the forthcoming analysis, the properties of the primal functions
highlighted by Lemma~\ref{lemma:primal_function_properties} will
be linked to the cost estimation mechanism.

\subsection{Uniform convergence of estimated subgradients}
In Section~\ref{sec:distributed_algorithm} we have seen that, being the cost function
unknown, each agent uses a surrogate function in problem~\eqref{eq:alg_local_prob}
in place of the actual objective. Next we provide a sequence of results that show that
the estimated subgradients of the primal functions converge to the true ones.
This fact will be necessary in the proof of Theorem~\ref{thm:convergence}
to assess that Algorithm~\ref{alg:algorithm} can asymptotically find
an optimal allocation of problem~\eqref{eq:primal_decomp_master}.
Similarly to the definition of the primal function~\eqref{eq:primal_decomp_subprob},
let us define $p_i^t(y_i)$ as the optimal cost of the subproblem with the surrogate
function at time $t$ for a given allocation $y_i \in \real$, i.e.,
\begin{align}
\begin{split}
  p_i^t(y_i) \triangleq \min_{x_i, \rho_i} \: & \: f_i^t(x_i) + M \rho_i
  \\
  \subj \: & \: g_i(x_i) \le y_i + \rho_i
  \\
  & \: x_i \in X_i, \:\: \rho_i \ge 0.
\end{split}
\label{eq:primal_decomp_subproblem_est}
\end{align}
We will work with the dual problems associated
to problems~\eqref{eq:primal_decomp_subprob} and~\eqref{eq:primal_decomp_subproblem_est}.
Let us derive the dual problem associated to~\eqref{eq:primal_decomp_subproblem_est}
(the procedure for problem~\eqref{eq:primal_decomp_subprob} follows similar arguments).
Let us compute the dual function,
\begin{align*}
  q_i^t(\mu_i)
  &= \!\inf_{x_i \in X_i, \: \rho_i \ge 0} \! \big[ f_i^t(x_i) + M \rho_i + \mu_i \big( g_i(x_i) - y_i - \rho_i \big) \big]
  \\
  &= \begin{cases}
    \displaystyle\min_{x_i \in X_i} \big[ f_i^t(x_i) + \mu_i \big( g_i(x_i) - y_i \big)\big]
      & \text{if } \mu_i \le M
    \\
    -\infty & \text{otherwise},
  \end{cases}
\end{align*}
where we replaced $\inf$ with $\min$ since $X_i$ is compact.
Note that $\mu_i \le M$ is the domain associated to the dual function.
The dual problem associated to~\eqref{eq:primal_decomp_subproblem_est} is
\begin{align}
\begin{split}
  \max_{\mu_i} \: & \: q_i^t(\mu_i)
  \\
  \subj \: & \: 0 \le \mu \le M,
\end{split}
\label{eq:dual_prob_est}
\end{align}
for all $i \in \until{N}$. Note that for all $y_i \in \real$ the
function $q_i^t$ is continuous and thus the maximum in~\eqref{eq:dual_prob_est}
exists finite. In a similar way, the dual problem associated
to~\eqref{eq:primal_decomp_subprob} is exactly as problem~\eqref{eq:dual_prob_est},
except that $q_i^t$ is replaced with the dual function associated
to~\eqref{eq:primal_decomp_subprob}, i.e.,
\begin{align*}
    q_i(\mu_i)
  = \begin{cases}
    \displaystyle\min_{x_i \in X_i} \big[ f_i(x_i) + \mu_i \big( g_i(x_i) - y_i \big)\big]
      & \text{if } \mu_i \le M
    \\
    -\infty & \text{otherwise}.
  \end{cases}
\end{align*}
\begin{lemma}\label{lemma:dual_function_convergence}
  Let Assumptions~\ref{ass:problem} and~\ref{ass:oracles}
  hold. Then, the dual function sequence $\{q_i^t\}_t$ converges
  to $q_i$, uniformly in $\mu_i \in D_i = \{ \mu_i : \mu_i \le M \}$
  and $y_i \in \real$, where $D_i$ is the domain of $q_i$ and $q_i^t$.
\end{lemma}
\begin{proof}
  Since our aim is to prove uniformity
  with respect to both $\mu$ and $y$, in this proof we denote
  the functions as $q^t(\mu, y)$ and $q(\mu, y)$ to show explicitly
  the dependence of $q^t$ and $q$ on both $\mu$ and $y$.
  By definition, we have that, for any fixed $\mu \in D$ and $y \in \real$,
  \begin{align}
    q(\mu,y)
    &= \min_{x \in X} \big( f(x) + \mu (g(x) - y) \big)
    \nonumber
    \\
    &
    \le f(x) + \mu (g(x) - y),
    \hspace{0.5cm}
    \text{for all } x \in X,
  \label{eq:dual_fn_dim}
  \end{align}
  and also
  \begin{align}
    q^t(\mu,y)
    &= \min_{x \in X} \big( f^t(x) + \mu (g(x) - y) \big)
    \nonumber
    \\
    &
    \le f^t(x) + \mu (g(x) - y)
    \hspace{0.5cm}
    \text{for all } x \in X.
  \label{eq:dual_fn_n_dim}
  \end{align}
  By the uniform convergence of $\{f^t\}_t$ (cf. Assumption~\ref{ass:oracles}),
  we have that for all $\varepsilon > 0$ there exists $N > 0$ such that for all
  $t\geq N$ it holds
  \begin{align*}
    f^t(x) - f(x) < \varepsilon
    \hspace{0.5cm} \text{and} \hspace{0.5cm}
    f(x) - f^t(x) < \varepsilon,
  \end{align*}
  for all $x \in X$.
  Subtracting~\eqref{eq:dual_fn_dim} from~\eqref{eq:dual_fn_n_dim} and using
  the uniform convergence we obtain, for all $t \ge N$ and for any $\mu \in D$
  and $y \in \real$,
  \begin{align*}
    q^t(\mu,y) - q(\mu,y)
    &\le f^t(x) - f(x) \hspace{0.7cm} \text{for all } x \in X
    \nonumber
    \\
    &< \varepsilon.
  \end{align*}
  Similarly, subtracting~\eqref{eq:dual_fn_n_dim} from~\eqref{eq:dual_fn_dim},
  we obtain, for all $t \ge N$ and for any $\mu \in D$ and $y \in \real$,
  \begin{align*}
    q(\mu, y) - q^t(\mu, y) < \varepsilon.
  \end{align*}
  Since the previous results do not actually depend on the chosen $\mu$ or $y$, they are
  uniform in $\mu$ and $y$. Therefore we have proven that for all $\varepsilon > 0$ there exists
  $N \ge 0$ such that $|q^t(\mu) - q(\mu)| < \varepsilon$ for all $t \ge N$,
  uniformly in $y \in \real$ and $\mu \in D$.
\end{proof}

Owing to Lemma~\ref{lemma:primal_function_subgradient},
a subderivative of the time-varying primal function $p_i^t$ at any
$y_i \in \real$ is given by ${p_i^t}^\prime(y_i) = -\mu_i^t(y_i)$,
where $\mu_i^t(y_i)$ is a maximum of $q_i^t(\cdot, y_i)$ (with respect
to $\mu_i$) in the interval $0 \le \mu_i \le M$.
Since there may be several maxima,
we now introduce a tie-break rule to make the maximum unique (later it will
be formalized specifically for the algorithm, cf. Assumption~\ref{ass:tiebreak_rule}).
In particular, we assume that among all the maxima we always select the smallest
one, i.e., we define $\mu_i^t(y_i)$ as the function
\begin{align}
  \mu_i^t(y_i) \triangleq \min\Big\{ \argmax_{0 \le \mu_i \le M} \: q_i^t(\mu_i, y_i) \Big\},
\label{eq:tie_break_rule}
\end{align}
where here we intend the outer minimization as the selection of the smallest
number from the set of maxima returned by the $\argmax$ operator.
With this definition at hand, the subderivative ${p_i^t}^\prime$ defined above
is a well-defined function of $y_i$.
A similar definition holds for the subderivative of $p_i(y_i)$, %
\begin{align*}
  p_i^\prime(y_i) = -\mu_i(y_i) \triangleq -\min\Big\{ \argmax_{0 \le \mu_i \le M} \: q_i(\mu_i, y_i) \Big\}.
\end{align*}

\begin{lemma}\label{lemma:uniform_convergence_pi}
  Let Assumption \ref{ass:oracles} hold.
  Then, the subderivative function ${p_i^t}^\prime(y_i)$ converges uniformly to
  $p_i^\prime(y_i)$, i.e.
  for all $\eta > 0$ there exists $N > 0$ such that $|{p_i^t}^\prime(y_i) - p_i^\prime(y_i)| < \eta$
  for all $t \ge N$ and $y_i \in \real$.
\end{lemma}
\begin{proof}
  To ease the notation, we drop the index $i$.
  Since ${p^t}^\prime(y) = -\mu^t(y)$,
  we need to prove that the function
  $\mu^t(y)$ converges uniformly to $\mu(y)$.
  By definition~\eqref{eq:tie_break_rule}, the function sequence $\{\mu^t(y)\}_{t \in \natural}$
  is uniformly bounded in $[0, M]$. Thus we can extract a convergent
  subsequence $\{\mu^{t_n}(y)\}_{n \in \natural}$ and denote by
  $\bar{\mu}(y)$ its limit function. Let us first show that the limit function
  maps each $y$ to a maximum of $q(\mu, y)$ over $\mu \in [0,M]$.
  For all $y \in \real$, by optimality of $\mu^{t_n}(y)$ for $q^{t_n}$ it holds
  \begin{align*}
    q^{t_n}(\mu(y), y) \le q^{t_n}(\mu^{t_n}(y), y),
    \hspace{1cm}
    \forall n \in \natural.
  \end{align*}
  By taking the limit as $n \to \infty$ and by using
  Lemma~\ref{lemma:dual_function_convergence}, we obtain for all
  $y \in \real$
  \begin{align*}
    q(\mu(y), y) \le q(\bar{\mu}(y), y).
  \end{align*}
  However, by optimality of $\mu(y)$ for $q$ it also holds
  $q(\mu(y), y) \ge q(\bar{\mu}(y), y)$ for all $y \in \real$.
  Thus, equality follows for all $y$ and therefore
  \begin{align}
    \bar{\mu}(y) \in \argmax_{\mu \in [0,M]} \: q(\mu, y),
    \hspace{1cm}
    \forall y \in \real.
  \label{eq:mu_bar_proof}
  \end{align}
  We finally need to show that $\bar{\mu}(y)$ is also the smallest number
  in $\argmax_{\mu \in [0,M]} \: q(\mu, y)$.
  Let us denote
  \begin{align*}
    Q^\star(y) &= \argmax_{0 \le \mu \le M} \: q(\mu, y),
    \\
    Q^t(y) &= \argmax_{0 \le \mu \le M} \: q^t(\mu, y).
  \end{align*}
  By~\eqref{eq:tie_break_rule}, for all $y \in \real$ it holds
  \begin{align*}
    \mu^{t_n}(y) \le \mu,
    \hspace{1cm}
    \forall \mu \in Q^t{t_n}(y).
  \end{align*}
  By taking the limit as $n$ goes to infinity and by
  using~\eqref{eq:mu_bar_proof}, we obtain for all $y \in \real$
  \begin{align*}
    \bar{\mu}(y) \le \mu,
    \hspace{1cm}
    \forall \mu \in Q^\star(y),
  \end{align*}
  and the proof follows.
\end{proof}

\subsection{Estimated subgradients are epsilon-subgradients}
The uniform convergence of the the estimated subgradients is not
enough to prove that Algorithm~\ref{alg:algorithm}
is able to asymptotically recover optimality.
However, it turns out that the subgradients of the time-varying primal
functions $p_i^t$ are so-called $\epsilon$-subgradients of
the true primal function $p_i$.
Formally, given a convex function $\map{\varphi(\theta)}{\real^n}{\real}$,
an $\epsilon$-subgradient of $\varphi$ at some $\theta_0 \in \real^n$,
is a vector $\subgrad_\epsilon \varphi(\theta_0) \in \real^n$ satisfying
\begin{align*}
  \varphi(\theta)
  \ge \varphi(\theta_0) + \subgrad_\epsilon \varphi(\theta_0)^\top (\theta - \theta_0) - \epsilon,
  \hspace{1cm}
  \forall \theta \in \real^n,
\end{align*}

In the following important proposition, we prove a central result for the
analysis.
\begin{proposition}
  Let Assumptions~\ref{ass:problem},~\ref{ass:slater} and~\ref{ass:oracles} hold.
  Then, there exists a sequence $\{\epsilon_i^t\}_{t \in \natural}$ of non-negative scalars
  such that for all $t \in \natural$ and $y_i \in \real$ it holds
  \begin{align}
    p_i(z) \ge p_i(y_i) + (z - y_i)^\top   {p_i^t}^\prime(y) - \epsilon_i^t,
    \hspace{0.7cm}
    \forall \: z \in \real.
  \label{eq:epsilon_subg_thm1}
  \end{align}
  Moreover, $\displaystyle\lim_{t \to \infty} \epsilon_i^t = 0$.
\label{prop:epsilon_subgr}
\end{proposition}
\begin{proof}
  To ease the notation, we drop the index $i$.
  We begin by proving that, for each fixed $t$ and $y \in \real$,
  there exists a finite number $\epsilon^t$
  satisfying~\eqref{eq:epsilon_subg_thm1}.
  We will then find an upper bound of $\epsilon^t$, independent of
  $y$, that goes to zero as $t$ goes to infinity, which yields the
  desired result.

  Fix $t \in \natural$ and $y_0\in \real$.
  Let us define $\epsilon^t(y_0)$ as the smallest non-negative number
  satisfying~\eqref{eq:epsilon_subg_thm1} at $y_0$, i.e.
  \begin{align}
  \begin{split}
    \epsilon^t(y_0) \!=\! \inf_{\epsilon} \:
    & \epsilon
    \\
    \subj \: &
    \epsilon \ge p(y_0) \!+\! (z\!-\!y_0) {p^t}^\prime(y_0) \!-\! p(z), \:
    \forall z \in \real.
  \end{split}
  \label{eq:epsilon_star}
  \end{align}
 By definition, $\epsilon^t(y_0) \geq 0$.
 We must prove that 
  the infimum in~\eqref{eq:epsilon_star} is attained
  at a real number (i.e. that $\epsilon^t(y_0) \ne +\infty$).
  The optimization problem~\eqref{eq:epsilon_star} is in
  epigraph form and can be equivalently rewritten as
  \begin{subequations}
  \begin{align}
     \epsilon^t(y_0) = \sup_{z \in \real} \big[p(y_0) + (z-y_0) {p^t}^\prime(y_0) - p(z)\big] .
  \label{eq:epsilon_star_comp}
  \end{align}
  Using the properties of the $\sup$, we can rewrite $\epsilon^t(y_0)$ as
  \begin{align}
    \epsilon^t(y_0)
    = -\inf_{z \in \real} r(z),
  \label{eq:delta_star}
  \end{align}
  \end{subequations}
  with
  \begin{align}
    r(z) = p(z) - z {p^t}^\prime(y_0) - p(y_0) + y_0 {p^t}^\prime(y_0).
  \label{eq:r_definition}
  \end{align}
  Note that $r(z)$ is also a function of $t$ and $y_0$, however
  we leave these arguments as implicit so as to keep the notation light.
  We now show that the minimum of $r(z)$ exists, which in turn
  implies that $\epsilon^t(y_0) \in \real$.
  To see this, first note that since $p(z)$ is convex then also $r(z)$ is convex.
  Let us study the subderivative of $r(z)$.
  By Lemma~\ref{lemma:primal_function_properties}, $p(y)$ is linear for
  $y \le \ymin$ and for $y \ge \ymax$.
  Thus, $r(z)$ is differentiable for all $z\in (-\infty, \ymin) \cup (\ymax, +\infty)$.
  For all $z \le \ymin$, it holds
  \begin{align*}
    r^\prime(z)
    = p^\prime(z) - {p^t}^\prime(y_0)
    =  -M - {p^t}^\prime(y_0) \le 0,
  \end{align*}
  where the last equality follows by Lemma~\ref{lemma:primal_function_properties} and the
  inequality follows by Lemma~\ref{lemma:p_bounded_negative_subgrad}.
  Analogously, for all $z \ge \ymax$, it holds
  \begin{align*}
    r^\prime(z)
    = p^\prime(z) - {p^t}^\prime(y_0)
    =  - {p^t}^\prime(y_0) \ge 0.
  \end{align*}
  Thus $r(z)$ is non increasing for $z \le \ymin$
  and non decreasing for $z \ge \ymax$.
  Being the function convex (and thus continuous), there
  exists a (finite) minimum in the interval $[\ymin, \ymax]$.
  i.e.,
  \begin{align*}
    \min_{z \in \real} r(z)
    &= \min_{z \in [\ymin, \ymax]} r(z).
  \end{align*}
  Thus, $\epsilon^t(y_0) \in \real$ since the $\inf$ in \eqref{eq:epsilon_star}
  is finite.

  Now we proceed to %
  compute a vanishing
  overestimate of $\epsilon^t(y_0)$.
  Consider the sequence
  $\{{p^t}^\prime(y_0)\}_{t \in \natural}$ and fix $\beta > 0$.
  Define $\eta = \beta / |\ymin - \ymax| > 0$. By
  Lemma~\ref{lemma:uniform_convergence_pi}, %
  there exists $N > 0$ such that $|{p^t}^\prime (y_0) - p^\prime(y_0)| < \eta$
  for all $t \ge N$.
  To compute the overestimate of $\epsilon^t(y_0)$, we replace
  $p(z)$ with a convex, piece-wise linear underestimate
  $p_{y_0}^\textsc{aux}(z)$, defined as
  \begin{align*}
    p_{y_0}^\textsc{aux}(z)
    = \max \Big\{&
      p(\ymin) + M (\ymin - y),
    \\
      &p(y_0) + p^\prime(y_0) (y - y_0) , \:\:\:
      p(\ymax) \Big\}.
  \end{align*}
  The resulting function consists of three pieces.
  The left-most piece and the right-most piece are obtained
  by prolonging the two lateral linear pieces of $p(z)$ inside the
  interval $[\ymin, \ymax]$, while the central piece
  is the tangent line crossing $p(z)$ at $y_0$. By construction,
  this function satisfies $p_{y_0}^\textsc{aux}(z) \le p(z)$ for all $z \in \real$.
  Let us compute the break points, which we denote by $y^L$ and $y^R$
  (see Figure~\ref{fig:p_aux}).
  To compute $y^L$, we must intersect the first two pieces, i.e.
  \begin{align*}
    \underbrace{
      -M y^L + p(\ymin) + M \ymin
    }_\text{left piece}
    =
    \underbrace{
      p(y_0) + (y^L - y_0) p^\prime(y_0)
    }_\text{central piece},
  \end{align*}
  which results in
  \begin{align}
    y^L = \frac{p(\ymin) + M \ymin - p(y_0) + y_0 p^\prime(y_0)}{p^\prime(y_0) + M}.
  \label{eq:y_L_comp}
  \end{align}
  Similarly, we can compute $y^R$, which is equal to
  \begin{align}
    y^R = y_0 + \frac{p(\ymax) - p(y_0)}{p^\prime(y_0)}.
  \label{eq:y_R_comp}
  \end{align}
  Notice that the value of $p_{y_0}^\textsc{aux}(y^L)$ is equal to $p(\ymax)$.
  
   \begin{figure}[htbp]\centering
	  \includegraphics[scale=1]{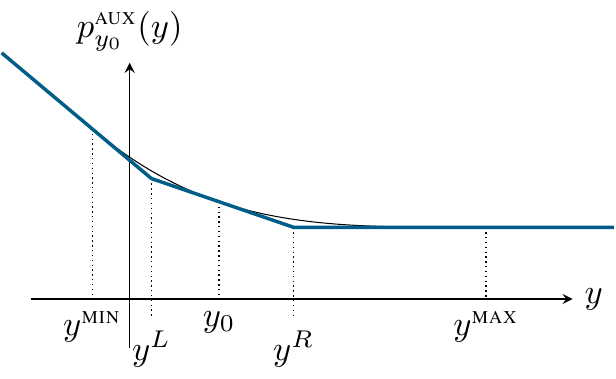}
	  \caption{
	    Graphical representation of the surrogate function $p_{y_0}^\textsc{aux}(z)$ (in blue).
	    The original primal function $p(z)$ is the black curve near the blue one.
	  }
	\label{fig:p_aux}
	\end{figure}

  Now, %
  similarly to $r(z)$,
  let us define functions $r^{\textsc{aux},t}(z)$ corresponding to $p^t(z)$
  for all $t$.
  Then,
  we use them to compute the upper bound on $\epsilon^t(y_0)$,
  in a similar way as in~\eqref{eq:epsilon_star_comp}--\eqref{eq:delta_star}.
   The functions $r^{\textsc{aux},t}(z)$ are defined as
  \begin{align}
    r^{\textsc{aux},t}(z) = p_{y_0}^\textsc{aux}(z) - z {p^t}^\prime(y_0) - p(y_0) + y_0 {p^t}^\prime(y_0),
  \label{eq:r_aux}
  \end{align}
  for all $t \ge 0$.
  As before, these functions also depend on $y_0$, which is
  omitted in the notation because it is fixed.
  It holds $r^{\textsc{aux},t}(z) \le r(z)$ (since $p_{y_0}^\textsc{aux}(z) \le p(z)$).
  Being $p_{y_0}^\textsc{aux}(z)$ piece-wise linear with three pieces, then also $r^{\textsc{aux},t}(z)$
  is piece-wise linear with three pieces. Similarly to~\eqref{eq:delta_star}, let us now
  define the overestimate of $\epsilon^t(y_0)$ as
  \begin{align*}
    \epsilon^{\textsc{aux},t}(y_0)
    \triangleq -\inf_{z \in \real} r^{\textsc{aux},t}(z)
    = -\min_{z \in [y^L, y^R]} r^{\textsc{aux},t}(z),
  \end{align*}
  where the equality holds since the function $r^{\textsc{aux},t}(z)$
  admits minimum in $[y^L, y^R]$ (by following the same reasoning
  used for $r(z)$).
  Since $r^{\textsc{aux},t}(z) \le r(z)$ for all $z$, the same holds
  for the minimum of such functions over $[y^L, y^R]$,
  from which we see that indeed
  $\epsilon^{\textsc{aux},t}(y_0) \ge \epsilon^t(y_0)$.
  Since $r^{\textsc{aux},t}(z)$ is linear in the interval $[y^L, y^R]$,
  the minimum is attained either at $y^L$ or at $y^R$:
  \begin{align}
    \epsilon^{\textsc{aux},t}(y_0) = - \min\Big\{ r^{\textsc{aux},t}(y^L), \:\: r^{\textsc{aux},t}(y^R) \Big\}.
  \label{eq:epsilon_aux_min}
  \end{align}
  Let us compute the value of the function at $y^L$ and $y^R$, i.e.
  \begin{align*}
    r^{\textsc{aux},t}(y^L)
    &=
    p_{y_0}^\textsc{aux}(y^L) - y^L {p^t}^\prime(y_0) - p(y_0) + y_0 {p^t}^\prime(y_0)
    \\
    &=
    p(y_0) + p^\prime(y_0) (y^L - y_0) - y^L {p^t}^\prime(y_0)
    \\
    &\hspace{0.5cm}- p(y_0) + y_0 {p^t}^\prime(y_0)
    \\
    &=
    (p^\prime(y_0) - {p^t}^\prime(y_0)) (y^L - y_0),
  \end{align*}
  and, similarly,
  \begin{align*}
    r^{\textsc{aux},t}(y^R)
    &=
    (p^\prime(y_0) - {p^t}^\prime(y_0)) (y^R- y_0),
  \end{align*}
  which always have opposite sign since $y^L \le y_0 \le y^R$.
  Thus, we can distinguish two cases.
  If $r^{\textsc{aux},t}(y^L) \le 0$, then the minimum in~\eqref{eq:epsilon_aux_min}
  is attained at $r^{\textsc{aux},t}(y^L)$ and therefore
  \begin{align*}
    \epsilon^{\textsc{aux},t}(y_0)
    =
    - r^{\textsc{aux},t}(y^L)
    &=
    \underbrace{-(p^\prime(y_0) - {p^t}^\prime(y_0)) (y^L - y_0)}_{\ge 0}
    \\
    &= \underbrace{|p^\prime(y_0) - {p^t}^\prime(y_0)|}_{\le \eta \:\: \forall t \ge N} \underbrace{|y^L - y_0|}_{\le |\ymin - \ymax|}
    \\
    & \le \eta |\ymin - \ymax|, \hspace{1cm} \forall t \ge N.
  \end{align*}
  Likewise, if $r^{\textsc{aux},t}(y^R) \le 0$, we obtain
  \begin{align*}
    \epsilon^{\textsc{aux},t}(y_0)
    =
    - r^{\textsc{aux},t}(y^R)
    =
    -(p^\prime(y_0) - {p^t}^\prime(y_0)) (y^R - y_0)
    \\
    \le \eta |\ymin - \ymax|, \hspace{0.5cm} \forall t \ge N.
  \end{align*}
  In either cases, it holds
  \begin{align*}
    \epsilon^t(y_0) \le \epsilon^{\textsc{aux},t}(y_0) \le \underbrace{\eta |\ymin - \ymax|}_{\beta}, \hspace{1cm} \forall t \ge N.
  \end{align*}
  which is independent of the chosen $y_0$. Thus we conclude
  \begin{align}
    0 \le \epsilon^t \le \max_{y \in \real} \epsilon^t(y) \le \max_{y \in \real} \epsilon^{\textsc{aux},t}(y) \le \beta, \hspace{0.5cm} \forall t \ge N.
  \label{eq:prop_epsilon_subg_final}
  \end{align}
  Since $\beta > 0$ is arbitrary, it follows that
  $\displaystyle\lim_{t \to \infty} \epsilon^t = 0$.
\end{proof}

Note that, in order for Proposition~\ref{prop:epsilon_subgr} to hold,
Assumption~\ref{ass:oracles} is important. Indeed, if Assumption~\ref{ass:oracles}
does not hold, it can be seen that in the previous proof that
Lemma~\ref{lemma:uniform_convergence_pi} could not be applied, and thus
one could use the fact that $|{p^t}^\prime (y_0) - p^\prime(y_0)| < \eta$.
As a consequence,~\eqref{eq:prop_epsilon_subg_final} would not be valid
and it would not be possible to conclude that $\lim_{t \to \infty} \epsilon^t = 0$.

\section{Algorithm Analysis: Convergence Result}
\label{sec:analysis}

In this section, we provide the main theoretical result
for the convergence analysis of the distributed algorithm.
The line of proof is based on the ideas in~\cite{camisa2020distributed},
however we will need to make the necessary modifications to the analysis
since in the local problems~\eqref{eq:alg_local_prob} we replaced
the true cost function with an estimated version.
Before introducing the main theoretical result, let us recall
the needed results.

\subsection{Unconstrained Formulation of Master Problem}
In order to analyze Algorithm~\ref{alg:algorithm}, we proceed to
perform a graph-induced reformulation of the master
problem~\eqref{eq:primal_decomp_master} that makes it
amenable to distributed computation.
Formally, consider the communication graph $\GG = (V, \EE)$.
For all edges $(i,j) \in \EE$, let $z_{ij} \in \real$ be a vector associated
to the edge $(i,j)$ and denote by $z \in \real^{|\EE|}$ the vector stacking
all $z_{ij}$.
Consider the change of coordinates for problem~\eqref{eq:primal_decomp_master}
defined through the following linear mapping
\begin{align}
  y_i
  = \sum_{j \in \nbrs_i} \! (z_{ij} - z_{ji}) + \frac{b}{N},
  \hspace{0.5cm}
  \forall \: i \in \until{N}.
\label{eq:change_of_coordinate}
\end{align}
The main point in introducing these new variables is that
they implicitly encode the constraint of the master problem, i.e.,
\begin{align*}
  \sum_{i=1}^N y_i = \sum_{i=1}^N \sum_{j \in \nbrs_i} \! (z_{ij} - z_{ji}) + b = b,
\end{align*}
which follows by the assumption that $\GG$ is undirected.
Let us apply this change of coordinates to
problem~\eqref{eq:primal_decomp_master}. Formally, for all $i \in \until{N}$,
define the functions
\begin{align*}
  \tp_i \big( \{z_{ij}, z_{ji}\}_{j \in \nbrs_i} \big)
  \triangleq
  p_i\bigg[ \sum_{j \in \nbrs_i} \! (z_{ij} - z_{ji}) + \frac{b}{N} \bigg],
  \hspace{0.2cm}
  z \in \real^{|\EE|}.
\end{align*}
In the next lemma we recall the formal equivalence
of the master problem with its unconstrained version.
\begin{lemma}[{\cite[Corollary 4.2]{camisa2020distributed}}]
\label{lemma:equivalence_master_prob_z}
  Problem~\eqref{eq:primal_decomp_master} is equivalent to the
  unconstrained optimization problem
  \begin{align}
	  \min_{z \in \real{|\EE|}} \:
	    \sum_{i=1}^N
	    \tp_i \big( \{z_{ij}, z_{ji}\}_{j \in \nbrs_i} \big),
	\label{eq:problem_z}
	\end{align}
	in the sense that \emph{(i)} the optimal costs are equal, and
  \emph{(ii)} if $\{z_{(ij)}^\star\}_{(i,j) \in \EE}$ is optimal
  for~\eqref{eq:problem_z}, then $y_i^\star = (z_{ij}^\star - z_{ji}^\star) + b/N$
  for all $i \in \until{N}$ is an optimal solution of~\eqref{eq:primal_decomp_master}.
	\oprocend
\end{lemma}

In the following, we denote the cost function of~\eqref{eq:problem_z}
as $\tp(z) = \sum_{i=1}^N \tp_i \big( \{z_{ij}, z_{ji}\}_{j \in \nbrs_i} \big)$.

\subsection{Convergence Theorem}

We are now ready to formulate the main theoretical result.
We formalize the assumption on the tie-break rule, stating that among
all the Lagrange multipliers of problem \eqref{eq:alg_local_prob}
we choose the smallest one.
\begin{assumption}[Tie-break rule]
\label{ass:tiebreak_rule}
  At each iteration $t \ge 0$, each agent $i \in \until{N}$
  selects $\mu_i^t$ as the smallest Lagrange multiplier of
  problem~\eqref{eq:alg_local_prob}.
\end{assumption}
As regards the step-size, we make the following standard assumption.
\begin{assumption}
  \label{ass:stepsize}
  The step-size sequence $\{\alpha^t\}_{t \ge 0}$, with each $\alpha^t \ge 0$,
  satisfies $\sum_{t=0}^\infty \alpha^t \!=\! \infty$ and
  $\sum_{t=0}^\infty (\alpha^t)^2 \!< \!\infty$.\oprocend
\end{assumption}
The convergence theorem is reported next.
\begin{theorem}
\label{thm:convergence}
  Let Assumptions~\ref{ass:problem},~\ref{ass:slater},~\ref{ass:oracles},~\ref{ass:stepsize}
	and~\ref{ass:tiebreak_rule} hold.
  Then, assuming the allocation
  vectors $y_i^0$ are initialized such that
  $\sum_{i=1}^N y_i^0 = b$,
  for a sufficiently large $M > 0$ the sequences
  $\{ x_i^t \}_{t \ge 0}$ and $\{ y_i^t \}_{t \ge 0}$
  generated by Algorithm~\ref{alg:algorithm} for all
  $i \in \until{N}$ are such that
  \begin{enumerate}
    \item[(i)] the sequence $\{(y_1^t, \ldots\, y_N^t)\}_{t \ge 0}$
      converges to an optimal solution of problem~\eqref{eq:primal_decomp_master};
    
    \item[(ii)] $\lim_{t \to \infty} \sum_{i=1}^N f_i^t(x_i^t) + M \rho_i^t = f^\star$,
      where $f^\star$ is the optimal cost of~\eqref{eq:problem_original};
    
    \item[(iii)] every limit point of $\{(x_1^t, \ldots\, x_N^t)\}_{t \ge 0}$
      is an optimal solution of~\eqref{eq:problem_original}.
  \end{enumerate}
\end{theorem}
\begin{proof}
  By Lemma~\ref{lemma:relaxation_primal_decomposition} and
  Lemma~\ref{lemma:equivalence_master_prob_z}, problem~\eqref{eq:problem_z}
  has the same optimal cost as problem~\eqref{eq:problem_original}.
  Recall that $p_i^t(y_i)$ denotes the optimal cost of
  problem~\eqref{eq:alg_local_prob} for all $i \in \until{N}$ and $t \ge 0$.
  Let us consider a subgradient method applied to problem~\eqref{eq:problem_z}.
  Instead of the standard subgradient method, we replace the subgradients
  of $\tp_i \big( \{z_{ij}^t, z_{ji}^t\}_{j \in \nbrs_i} \big)$ with a subgradient of
  $\tp_i^t \big( \{z_{ij}^t, z_{ji}^t\}_{j \in \nbrs_i} \big)
  \triangleq p_i^t[ \sum_{j \in \nbrs_i} \! (z_{ij}^t - z_{ji}^t) ]$
  and therefore consider the update
  \begin{align}
    z_{ij}^{t+1}
    & = z_{ij}^t - \alpha^t {\tp_i^t} \:\! ' \big( \{z_{ij}^t, z_{ji}^t\}_{j \in \nbrs_i} \big)
    \hspace{0.5cm}
    \forall (i,j) \in \EE,
  \label{eq:epsilon_subgradient_method}
  \end{align}
  initialized at some $z^0 \in \real^{|\EE|}$.
  As we will see in a moment, the update~\eqref{eq:epsilon_subgradient_method}
  is in fact an $\epsilon$-subgradient method applied to problem~\eqref{eq:problem_z}.
  By Lemma~\ref{lemma:primal_function_subgradient}, the subderivatives of
  $p_i^t$ at $\sum_{j \in \nbrs_i} (z_{ij}^t - z_{ji}^t)$ are equal to
  \begin{align*}
    {p_i^t}^\prime \bigg[\sum_{j \in \nbrs_i} (z_{ij}^t - z_{ji}^t) \bigg] =  -\mu_i^t,
  \end{align*}
  where $\mu_i^t$ is a Lagrange multiplier of problem~\eqref{eq:primal_decomp_subproblem_est}
  (with $y_i = \sum_{j \in \nbrs_i} (z_{ij} - z_{ji})$) associated to the constraint
  $g_i(x_i) \le y_i + \rho_i$. As shown in~\cite{camisa2020distributed},
  using the change of coordinates~\eqref{eq:change_of_coordinate} we can show that the
  subderivatives of $\tp_i^t$ are equal to
  \begin{align*}
    {\tp_i^t} \:\! ' \big( \{z_{ij}^t, z_{ji}^t\}_{j \in \nbrs_i} \big)
    =
    \mu_j^t - \mu_i^t,
    \hspace{1cm}
    \forall (i,j) \in \EE,
  \end{align*}
  from which it follows that the update~\eqref{eq:epsilon_subgradient_method}
  can be rewritten as
  \begin{align}
    z_{ij}^{t+1}
    & = z_{ij}^t + \alpha^t (\mu_j^t - \mu_i^t),
    \hspace{1cm}
    \forall (i,j) \in \EE.
  \label{eq:epsilon_subgradient_method_rewrite}
  \end{align}
  By Proposition~\ref{prop:epsilon_subgr}, together with
  Assumption~\ref{ass:tiebreak_rule} and
  Lemma~\ref{lemma:uniform_convergence_pi}, we finally see that the
  update~\eqref{eq:epsilon_subgradient_method_rewrite} is an $\epsilon$-subgradient
  method applied to problem~\eqref{eq:problem_z}, with $\epsilon^t = \displaystyle\max_i \epsilon_i^t$
  going to $0$ as $t$ goes to $\infty$. Thus, by following the arguments
  of~\cite[Section 3.3]{bertsekas2015convex}
  and by also using the boundedness of the subderivatives
  (Lemma~\ref{lemma:p_bounded_negative_subgrad})
  and Assumption~\ref{ass:stepsize},
  we conclude that the sequence $\{z^t\}_{t \ge 0}$ generated
  by~\eqref{eq:epsilon_subgradient_method_rewrite} converges to an optimal
  solution $z^\star$ of problem~\eqref{eq:problem_z}.
  
  Let us rewrite the update~\eqref{eq:epsilon_subgradient_method_rewrite}
  in terms of $y$ by using the change of
  coordinate~\eqref{eq:change_of_coordinate},
  \begin{align}
    y_i^{t+1}
    &= \sum_{i=1}^N \bigg[ \sum_{j \in \nbrs_i} (z_{ij}^t - z_{ji}^t) + b/N \bigg]
    \nonumber
    \\
    &= y_i^t + \alpha^t \sum_{j \in \nbrs_i^t} \big( \mu_i^t - \mu_j^t \big),
    \label{eq:epsilon_subgradient_method_y}
  \end{align}
  where we also used the fact that the graph is undirected and that
  $\sum_{i=1}^N y_i^t = b$ for all $t \ge 0$ (by induction).
  Note that~\eqref{eq:epsilon_subgradient_method_y} is exactly the algorithm
  update~\eqref{eq:alg_update}. Moreover, the sequence $\{y^t\}$
  is correctly initialized in Algorithm~\ref{alg:algorithm}, since
  by~\eqref{eq:epsilon_subgradient_method_y} it must hold
  $\sum_{i=1}^N y_i^0 = b$. Thus, we conclude that the sequence
  $\{y^t\}_{t \ge 0}$ converges to $y^\star$, with components equal to
  \begin{align*}
    y_i^\star = \sum_{j \in \nbrs_i} (z_{ij}^\star - z_{ji}^\star) + b/N,
    \hspace{1cm}
    i \in \until{N}.
  \end{align*}
  Thus, point \emph{(i)} follows by
  Lemma~\ref{lemma:equivalence_master_prob_z}.
  We have thus shown that the sequence $\{y^t\}_{t \ge 0}$ generated by
  Algorithm~\ref{alg:algorithm} converges to an optimal
  solution of problem~\eqref{eq:primal_decomp_master}.
  
  Having proven \emph{(i)}, point \emph{(ii)} can be proven with the
  same arguments as in~\cite[Theorem 2.5 (i)]{camisa2020distributed}.

  To prove \emph{(iii)}, 
Consider the primal sequence
$\{(x_1^t, \ldots, x_N^t, \rho_1^t, \ldots, \rho_N^t)\}_{t \ge 0}$ generated
by the Algorithm~\ref{alg:algorithm}.
By summing over $i \in \until{N}$ the inequality
$g_i(x_i^t) \le y_i^t + \rho_i^t \1$
(which holds by construction), it holds
\begin{align}
  \sum_{i=1}^N g_i(x_i^t)
  \le
  \sum_{i=1}^N y_i^t + \sum_{i=1}^N \rho_i^t \1
  =
  \sum_{i=1}^N \rho_i^t \1.
\label{eq:proof_primal_recovery_t}
\end{align}
Define $\rho^t = \sum_{i=1}^N \rho_i^t$.
By construction, the sequence $\{(x_1^t, \ldots, x_N^t, \rho^t)\}_{t \ge 0}$ is bounded
(as a consequence of point \emph{(ii)} and continuity of the functions $f_i^t(x_i)+ M\rho_i$), so that
there exists a sub-sequence of indices $\{t_h\}_{h \ge 0} \subseteq \{t\}_{t \ge 0}$
such that the sequence $\{(x_1^{t_n}, \ldots, x_N^{t_h}, \rho^{t_h})\}_{h \ge 0}$
converges. Denote the limit point of such sequence as
$(\bar{x}_1, \ldots, \bar{x}_N, \bar{\rho})$.
From point \emph{(ii)} of the theorem and by using the uniform
convergence of the objective functions (cf. Assumption~\ref{ass:oracles}),
it follows that
\begin{align*} %
  \sum_{i=1}^N f_i(\bar{x}_i) + M \bar{\rho} = f^\star.
\end{align*}
By Lemma~\ref{lemma:relaxation_primal_decomposition}, it must hold $\bar{\rho} = 0$.
As the functions $g_i$ are continuous, by taking the limit
in~\eqref{eq:proof_primal_recovery_t} as $h \to \infty$, with $t = t_h$, it holds
  $\sum_{i=1}^N g_i(\bar{x_i}) \le \bar{\rho} \1 = \0$.
Therefore, the point $(\bar{x}_1, \ldots, \bar{x}_N)$ is an optimal solution
of problem~\eqref{eq:problem_original}.
\end{proof}

\section{Numerical Computations}\label{sec:simulations}
In this section, we show numerical computations to corroborate
the theoretical results. 
To estimate the cost functions, we consider a method similar to the one in~\cite{simonetto2020smooth},
without all the machinery to guarantee smoothness and strong convexity.

Formally, fix an agent $i$ and consider $K \in \natural$ samples $z_i^1, \ldots, z_i^K \in X_i$.
To build the estimated function, we let the agent compute $\gamma_1, \ldots, \gamma_K \in \real^{n_i}$
by solving the feasibility problem
\begin{align}
\begin{split}
  \find \: &\: \gamma^1, \ldots, \gamma^K
  \\
  \subj \:
  & \: f_i(z_i^h) \!+\! (z_i^\ell \!-\! z_i^h)^\top \gamma^h \leq f_i(z_i^\ell),
  \\
  & \hspace{2cm} \text{for all } h,\ell = \interv{K}, \: h \ne \ell
\end{split}
\label{eq:feasibility_prob}
\end{align}
The purpose of problem~\eqref{eq:feasibility_prob} is to compute the slope of the linear pieces that build
up the estimation of $f_i$ and has $K$ variables and $K(K-1)$ convexity constraints.
With the solutions $(\gamma^1, \ldots, \gamma^K)$ at hand, the oracle returns the estimated functions in the following form:
\begin{align}
  f_i^K(x) = \max_{k \in \until{K}} \big\{ f_i(z_i^k) + (x - z_i^k)^\top \gamma^k \big\},
\label{eq:estimated_function}
\end{align}
for all $i \in \until{N}$.
Note that $f_i^K(\cdot)$ is a piece-wise linear function.
Moreover, $f_i^K$ is the pointwise maximum of a finite number of affine functions, its epigraph is a non-empty polyhedron, and hence $f_i^K$ is convex, closed and proper
(see Theorem 1 in \cite{taylor2017smooth}).

At each iteration, each agent $i$ collects a new sample of the domain. Initially the samples are independent and identically distributed in the whole domain. As more points are added to the model, we start to reduce the space where sampling new points into balls centered in the current approximated solution. In fact, restricting the sampling space, force the model to refine the surrogate function in those area containing the approximated solution.
We experimentally tested that, thanks to the latter expedient, it is possible to reduce the number of samples to keep in the memory.
In fact, if the cost function becomes interesting only near the minimum, from a certain point on, estimating the part of the function far from the minimum becomes irrelevant, and all samples far from the minimum can be canceled. By following this intuition, together with the fact that the sample space is shrinking more and more around the potential solution, the samples collected further away in terms of time can be removed from the model. This relieves the function estimation, whose complexity is dependent on the number of points used.

We consider a network of $N=10$ agents in a 3-dimensional domain. 
We generate a random Erd\H{o}s-R\'enyi graph with edge probability $0.2$.
We consider quadratic local functions $f_i$ and linear local constraints $g_i$.
Figure \ref{fig:cost_evolution_original} shows the evolution of the algorithm in terms of cost error $ |f^\star - \sum_{i=1}^N f_i(x_i^t) | $.
Despite the use of surrogate cost functions, the objective value converges to the optimal cost of the original problem with known cost function, as we expected from the theoretical results. 

\begin{figure}[!htbp]
  \centering
  \includegraphics[scale=1]{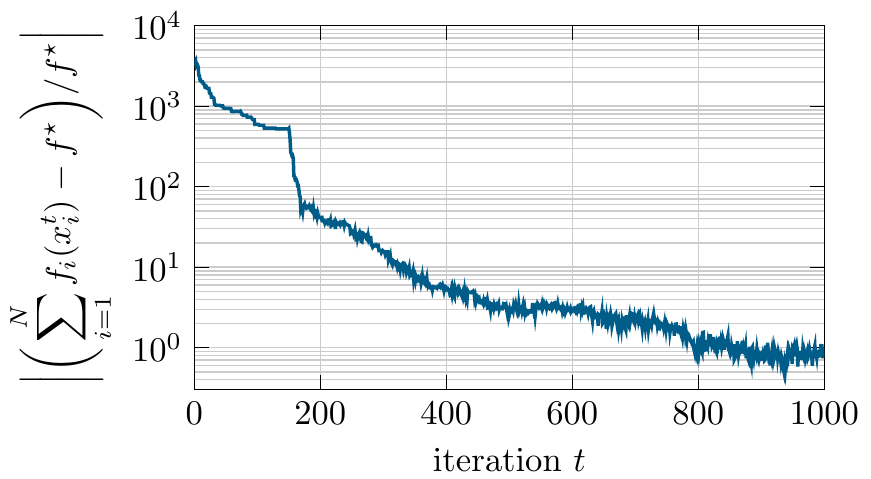}
  \caption{Evolution of the cost error in the numerical example.}
\label{fig:cost_evolution_original}
\end{figure}

\section{Conclusions}\label{sec:conclusions}

In this paper, we considered a challenging distributed optimization
scenario arising in several control problems of interest.
We focused on constraint-coupled optimization problems
with unknown cost functions to be learned and we proposed
a distributed optimization algorithm that only uses estimated
versions of the cost functions. We performed a thorough exploration
of the primal decomposition approach, by which concluded
that the distributed algorithm can be recast as an epsilon-subgradient
that asymptotically recovers consistency and provides an
optimal solution to the original problem.

\appendix
\subsection{Proof of Lemma~\ref{lemma:primal_function_subgradient}}
\label{app:proof_primal_func_subgradient}
  Fix an agent $i$. By Lemma~\ref{lemma:primal_function_subgradient},
  it holds $p_i^\prime(y_i) = -\mu_i(y_i)$, where $\mu_i(y_i)$
  is a Lagrange multiplier of problem~\eqref{eq:primal_decomp_subprob}
  associated to the constraint $g_i(x_i) \leq y_i + \rho_i$.
  Let us derive the dual function,
  \begin{align*}
    q_i(\mu_i)
    &= \inf_{x_i \in X_i, \: \rho_i \ge 0} \Big( f_i(x_i) + M \rho_i + \mu_i(g_i(x_i) - y_i - \rho_i) \Big)
    \\
    &=
    \begin{cases}
      \displaystyle\min_{x_i \in X_i} \: \Big(f_i(x_i) + \mu_i(g_i(x_i) - y_i)\Big)
        & \hspace{0.3cm}\text{if } \mu_i \le M,
      \\
      -\infty
        & \hspace{0.3cm}\text{otherwise}.
    \end{cases}
  \end{align*}
  Thus, the dual problem reads
  \begin{align*}
    \max_{0 \le \mu_i \le M} \: & \: q_i(\mu_i),
  \end{align*}
  from which we see that any optimal solution satisfies
  $0 \le \mu_i(y_i) \le M$. Therefore, it follows that
  $-M \le p_i^\prime(y_i) \le 0$.
  \oprocend

\bibliographystyle{IEEEtran}
\bibliography{references}

\begin{thebibliography}{10}
\providecommand{\url}[1]{#1}
\csname url@samestyle\endcsname
\providecommand{\newblock}{\relax}
\providecommand{\bibinfo}[2]{#2}
\providecommand{\BIBentrySTDinterwordspacing}{\spaceskip=0pt\relax}
\providecommand{\BIBentryALTinterwordstretchfactor}{4}
\providecommand{\BIBentryALTinterwordspacing}{\spaceskip=\fontdimen2\font plus
\BIBentryALTinterwordstretchfactor\fontdimen3\font minus
  \fontdimen4\font\relax}
\providecommand{\BIBforeignlanguage}[2]{{%
\expandafter\ifx\csname l@#1\endcsname\relax
\typeout{** WARNING: IEEEtran.bst: No hyphenation pattern has been}%
\typeout{** loaded for the language `#1'. Using the pattern for}%
\typeout{** the default language instead.}%
\else
\language=\csname l@#1\endcsname
\fi
#2}}
\providecommand{\BIBdecl}{\relax}
\BIBdecl

\bibitem{bullo2019lectures}
F.~Bullo, \emph{Lectures on network systems}.\hskip 1em plus 0.5em minus
  0.4em\relax Kindle Direct Publishing, 2019.

\bibitem{sasso2019interaction}
F.~Sasso, A.~Coluccia, and G.~Notarstefano, ``Interaction-based distributed
  learning in cyber-physical and social networks,'' \emph{IEEE Transactions on
  Automatic Control}, vol.~65, no.~1, pp. 223--236, 2019.

\bibitem{rabbat2005quantized}
M.~G. Rabbat and R.~D. Nowak, ``Quantized incremental algorithms for
  distributed optimization,'' \emph{IEEE Journal on Selected Areas in
  Communications}, vol.~23, no.~4, pp. 798--808, 2005.

\bibitem{bullo2009distributed}
F.~Bullo, J.~Cortes, and S.~Martinez, \emph{Distributed control of robotic
  networks: a mathematical approach to motion coordination algorithms}.\hskip
  1em plus 0.5em minus 0.4em\relax Princeton University Press, 2009, vol.~27.

\bibitem{cortes2017coordinated}
J.~Cort{\'e}s and M.~Egerstedt, ``Coordinated control of multi-robot systems: A
  survey,'' \emph{SICE Journal of Control, Measurement, and System
  Integration}, vol.~10, no.~6, pp. 495--503, 2017.

\bibitem{nedic2009distributed}
A.~Nedic and A.~Ozdaglar, ``Distributed subgradient methods for multi-agent
  optimization,'' \emph{IEEE Transactions on Automatic Control}, vol.~54,
  no.~1, pp. 48--61, 2009.

\bibitem{di2016next}
P.~Di~Lorenzo and G.~Scutari, ``Next: In-network nonconvex optimization,''
  \emph{IEEE Transactions on Signal and Information Processing over Networks},
  vol.~2, no.~2, pp. 120--136, 2016.

\bibitem{varagnolo2015newton}
D.~Varagnolo, F.~Zanella, A.~Cenedese, G.~Pillonetto, and L.~Schenato,
  ``Newton-raphson consensus for distributed convex optimization,'' \emph{IEEE
  Transactions on Automatic Control}, vol.~61, no.~4, pp. 994--1009, 2015.

\bibitem{mateos2017distributed}
D.~Mateos-N{\'u}nez and J.~Cort{\'e}s, ``Distributed saddle-point subgradient
  algorithms with {L}aplacian averaging,'' \emph{IEEE Transactions on Automatic
  Control}, vol.~62, no.~6, pp. 2720--2735, 2017.

\bibitem{notarstefano2019distributed}
G.~Notarstefano, I.~Notarnicola, and A.~Camisa, ``Distributed optimization for
  smart cyber-physical networks,'' \emph{Foundations and
  Trends{\textregistered} in Systems and Control}, vol.~7, no.~3, pp. 253--383,
  2019.

\bibitem{palomar2006tutorial}
D.~P. Palomar and M.~Chiang, ``A tutorial on decomposition methods for network
  utility maximization,'' \emph{IEEE Journal on Selected Areas in
  Communications}, vol.~24, no.~8, pp. 1439--1451, 2006.

\bibitem{giselsson2013accelerated}
P.~Giselsson, M.~D. Doan, T.~Keviczky, B.~De~Schutter, and A.~Rantzer,
  ``Accelerated gradient methods and dual decomposition in distributed model
  predictive control,'' \emph{Automatica}, vol.~49, no.~3, pp. 829--833, 2013.

\bibitem{notarnicola2019constraint}
I.~Notarnicola and G.~Notarstefano, ``Constraint-coupled distributed
  optimization: a relaxation and duality approach,'' \emph{IEEE Transactions on
  Control of Network Systems}, vol.~7, no.~1, pp. 483--492, 2019.

\bibitem{camisa2020distributed}
A.~Camisa, F.~Farina, I.~Notarnicola, and G.~Notarstefano, ``Distributed
  constraint-coupled optimization via primal decomposition over random
  time-varying graphs,'' \emph{arXiv preprint arXiv:2010.14489}, 2020.

\bibitem{dinh2013dual}
Q.~T. Dinh, I.~Necoara, and M.~Diehl, ``A dual decomposition algorithm for
  separable nonconvex optimization using the penalty function framework,'' in
  \emph{52nd IEEE Conference on Decision and Control}.\hskip 1em plus 0.5em
  minus 0.4em\relax IEEE, 2013, pp. 2372--2377.

\bibitem{tran2016fast}
Q.~Tran-Dinh, I.~Necoara, and M.~Diehl, ``Fast inexact decomposition algorithms
  for large-scale separable convex optimization,'' \emph{Optimization},
  vol.~65, no.~2, pp. 325--356, 2016.

\bibitem{simonetto2020smooth}
A.~Simonetto, ``Smooth strongly convex regression,'' \emph{arXiv preprint
  arXiv:2003.00771}, 2020.

\bibitem{williams2006gaussian}
C.~K. Williams and C.~E. Rasmussen, \emph{Gaussian processes for machine
  learning}.\hskip 1em plus 0.5em minus 0.4em\relax MIT press Cambridge, MA,
  2006, vol.~2, no.~3.

\bibitem{schulz2018tutorial}
E.~Schulz, M.~Speekenbrink, and A.~Krause, ``A tutorial on gaussian process
  regression: Modelling, exploring, and exploiting functions,'' \emph{Journal
  of Mathematical Psychology}, vol.~85, pp. 1--16, 2018.

\bibitem{todescato2017multi}
M.~Todescato, A.~Carron, R.~Carli, G.~Pillonetto, and L.~Schenato,
  ``Multi-robots gaussian estimation and coverage control: From client--server
  to peer-to-peer architectures,'' \emph{Automatica}, vol.~80, pp. 284--294,
  2017.

\bibitem{benevento2020multi}
A.~Benevento, M.~Santos, G.~Notarstefano, K.~Paynabar, M.~Bloch, and
  M.~Egerstedt, ``Multi-robot coordination for estimation and coverage of
  unknown spatial fields,'' in \emph{2020 IEEE International Conference on
  Robotics and Automation (ICRA)}.\hskip 1em plus 0.5em minus 0.4em\relax IEEE,
  2020, pp. 7740--7746.

\bibitem{camisa2019distributed}
A.~Camisa, F.~Farina, I.~Notarnicola, and G.~Notarstefano, ``Distributed
  constraint-coupled optimization over random time-varying graphs via primal
  decomposition and block subgradient approaches,'' in \emph{2019 IEEE 58th
  Conference on Decision and Control (CDC)}.\hskip 1em plus 0.5em minus
  0.4em\relax IEEE, 2019, pp. 6374--6379.

\bibitem{silverman1972primal}
G.~J. Silverman, ``Primal decomposition of mathematical programs by resource
  allocation: {I}--basic theory and a direction-finding procedure,''
  \emph{Operations Research}, vol.~20, no.~1, pp. 58--74, 1972.

\bibitem{bertsekas1999nonlinear}
D.~P. Bertsekas, \emph{Nonlinear programming}.\hskip 1em plus 0.5em minus
  0.4em\relax Athena Scientific, 1999.

\bibitem{bertsekas2015convex}
D.~P. Bertsekas and A.~Scientific, \emph{Convex optimization algorithms}.\hskip
  1em plus 0.5em minus 0.4em\relax Athena Scientific Belmont, 2015.

\bibitem{taylor2017smooth}
A.~B. Taylor, J.~M. Hendrickx, and F.~Glineur, ``Smooth strongly convex
  interpolation and exact worst-case performance of first-order methods,''
  \emph{Mathematical Programming}, vol. 161, no. 1-2, pp. 307--345, 2017.

\end{thebibliography}

\end{document}